 \documentclass[a4paper,11pt]{article}
\usepackage{latexsym,amsmath,amssymb,amsthm}
\theoremstyle{plain}
\newtheorem{theorem}{Theorem}[section]
\newtheorem{lemma}[theorem]{Lemma}

\newtheorem{corollary}[theorem]{Corollary}
\newtheorem{definition}[theorem]{Definition}

\newcommand{\seq}{\{x_n\}}

\date{}
\title{Compact Asymptotic Center and Common Fixed Point in Strictly
Convex Banach Spaces}
\author{Ali Abkar and Mohammad Eslamian}
\begin{document}
\maketitle Department of Mathematics, Imam Khomeini International
University, Qazvin 34149, Iran; email: abkar@ikiu.ac.ir,
mhmdeslamian@gmail.com

\begin{abstract}In this paper, we present some common fixed
point theorems for a commuting pair of mappings, including a
generalized nonexpansive singlevalued mapping and a generalized
nonexpansive multivalued mapping in strictly convex Banach spaces.
The results obtained in this paper extend and improve some recent
known results.\end{abstract}\maketitle \noindent Key Words: common
fixed point, generalized nonexpansive mapping, strictly convex
Banach space, asymptotic center
\section{Introduction}
The study of fixed points for multivalued contractions and
nonexpansive mappings using the Hausdorff metric was initiated by
Markin \cite{ma} and Nadler \cite{na}. Since then the metric fixed
point theory of multivalued mappings has been rapidly developed.
Using the Edelstein's method of asymptotic centers, Lim \cite{lim}
proved the existence of fixed points for multivalued nonexpansive
mappings on uniformly convex Banach spaces. Kirk and Massa \cite{km}
extended Lim's theorem to Banach spaces for which the asymptotic
center of a bounded sequence in a bounded closed convex subset is
nonempty and compact.\par On the other hand, in 2008, Suzuki
\cite{suz} introduced a condition on mappings, called condition (C),
which is weaker than nonexpansiveness and stronger than quasi
nonexpansiveness. He then proved some fixed point and convergence
theorems for such mappings. Motivated by this result, J.
Garcia-Falset, E. Llorens-Fuster and T. Suzuki in \cite{gf},
introduced two kinds of generalization for the condition (C) and
studied both the existence of fixed points and their asymptotic
behavior. Very recently, the current authors used a modified
condition for multivalued mappings, and proved some fixed point
theorems for multivalued mappings satisfying this condition in
Banach spaces \cite{a1,a2}, as well as in CAT(0) spaces \cite{a3}.
\par Let $D$ be a nonempty subset of a
strictly convex Banach space $X$, and let $T:D\to KC(D)$ be a
generalized nonexpansive multivalued mapping in the sense of
Garcia-Falset et al., that is, a mapping satisfying the conditions
(E) and $(C_\lambda)$ (see Definitions 2.9 and 2.10). Moreover, let
$t:D\to D$ be a single valued mapping satisfying the conditions (E)
and $(C_\lambda)$ of definitions 2.2 and 2.3. We call $t$ a
single-valued generalized nonexpansive mapping. We assume that $t$
and $T$ commute, and that the asymptotic center of a bounded
sequence in the fixed point set of $t$ is nonempty and compact. The
main result of this paper says that $t$ and $T$ have a common fixed
point (see Theorem 3.5). Our result improves a number of known
results; including that of Suzuki \cite{suz}, Garcia et al.
\cite{gf}, Kirk and Massa \cite{km}, Dhompongsa et al. \cite{dh},
and of Lim \cite{lim}.

\section{Preliminaries}
Let $X$ be a Banach space. $X$ is said to be strictly convex if
$\|x+y\|<2$ for all $x,y\in X,\, \|x\|=\|y\|=1$  and $x\neq y$. We
recall that a Banach space $X$ is said to be \emph{uniformly convex
in every direction} (UCDE, for short)  provided that for every
$\epsilon\in (0,2]$ and $z\in X$ with $\|z\|=1$, there exists a
positive number $\delta$ (depending on $\epsilon$ and $z$) such that
for all $ x,y\in X$ with $\|x\|\le 1,\, \|y\|\le 1,$ and $x-y\in\{
tz: t\in [-2, -\epsilon ]\cup [\epsilon, 2]\}$ we have $\|x+y\|\le
2(1-\delta)$. $X$ is said to be \emph{uniformly convex} if $X$ is
UCED and $inf\{\delta(\epsilon,z):\|z\|=1\}>0$ for all $\epsilon\in
(0,2]$. It is obvious that uniformly convexity implies UCED, and
UCED implies strictly convexity. \\The following definition is due
to Susuki \cite{suz}.

\begin{definition}(\cite{suz}) Let $T$ be a mapping on a subset $D$
of a  Banach space $X$. $T$ is said to satisfy condition (C) if
$$\frac{1}{2}\|x-Tx\|\le \|x-y\| \implies \|Tx-Ty\|\le \|x-y\|,\quad
x,y\in E.$$\end{definition} \noindent{\bf Example}  (\cite{suz}).
Define a mapping $T$ on [0,3] by $$T(x)=\begin{cases}0,&x \neq 3\\
1,& x=3.\end{cases}$$ Then $T$ satisfies the condition $(C)$, but
$T$ is not continuous, and hence $T$ is not nonexpansive.

 In \cite{gf}, J. Garcia-Falset et al.
introduced two generalizations of the condition (C) in a Banach
space:
\begin {definition}
 Let $T$ be a mapping on a subset $D$
of a Banach space $X$ and $\mu\geq1$. $T$ is said to satisfy
condition $(E_{\mu})$ if
$$\|x-Ty\|\le\mu \|x-Tx\|+\|x-y\|,\quad
x,y\in D.$$ We say that $T$ satisfies condition (E) whenever $T$
satisfies the condition $(E_{\mu})$ for some $\mu\geq1.$
\end{definition}
\begin{definition} Let $T$ be a
mapping on a subset $D$ of a Banach space $X$ and $\lambda\in
(0,1)$. $T$ is said to satisfy condition $(C_{\lambda})$ if
$$\lambda \|x-Tx\|\le \|x-y\| \implies \|Tx-Ty\|\le \|x-y\|,\quad
x,y\in D.$$ \end{definition} Notice that if
$0<\lambda_{1}<\lambda_{2}<1$ then the condition $(C_{\lambda_{1}})$
implies the condition $(C_{\lambda_{2}}).$ The following example
shows that the class of mappings satisfying the conditions (E) and
$(C_{\lambda})$ for some $\lambda\in(0,1)$ is
broader than the class  of mappings satisfying the condition (C).\\
\noindent{\bf Example} (\cite{gf}). For a given $ \lambda \in (0,1)$
define a mapping $T$ on [0,1] by
 $$T(x)=\begin{cases}\frac{x}{2},&x \neq 1\\ \frac{1+\lambda}{2+\lambda},&
x=1.\end{cases}$$ Then the mapping $T$ satisfies condition
$(C_{\lambda})$ but it fails condition $(C_{\lambda'})$ whenever $0
<\lambda<\lambda'.$ Moreover $T$ satisfies condition $(E_{\mu})$ for
$\mu=\frac{2+\lambda}{2}$.
\begin{theorem}(\cite {gf}) Let $D$ be a nonempty bounded convex subset of a Banach space $X$.
Let $T:D\to D$ satisfy the condition $(C_{\lambda})$ on $D$ for some
$\lambda\in(0,1)$. For $r\in  [\lambda,1)$ define a sequence
$\{x_{n}\}$ in $D$ by taking $x_{1}\in D$ and
$$x_{n+1}=rT(x_{n})+(1-r)x_{n} \qquad for\quad n\geq1,$$  then  $\{x_n\}$ is an
approximate fixed point sequence for $T$, that is
$$\lim_{n\to\infty} d(x_n,T(x_n))=0.$$
\end{theorem}
\begin{lemma}(\cite{gf})
Let $T$ be a mapping defined on a closed subset $D$ of a Banach
space $X$. Let $T$ be a single valued
 mapping satisfying the conditions (E) and $(C_{\lambda})$
for some  $\lambda\in(0,1)$. Then Fix(T) is closed. Moreover, if $X$
is strictly convex and $D$ is convex, then Fix(T) is also convex.
\end{lemma}
Let $D$ be a nonempty subset of $X$ Banach space $X$. for $x\in X$
denote$$dist(x,D)=inf\{\parallel x-z\parallel: z\in D\}.$$ We denote
by $CB(D)$ and $KC(D)$  the collection of all nonempty closed
bounded subsets, and nonempty compact convex subsets of $D$
respectively. The Hausdorff metric $H$ on $CB(X)$ is defined by
$$H(A,B):=\max \{\sup_{x\in A}dist(x,B),\sup_{y\in
B}dist(y,A)\},$$for all $A,B\in CB(X).$\\ Let $T:X\to 2^{X}$ be a
multivalued mapping. An element $x\in X$ is said to be a fixed point
of $T$, if  $x\in Tx$. \par It is obvious that if $D$ is a convex
subset of strictly convex Banach space $X$, then for $x\in X$, if
there exist $y,z\in D$ such that $$\parallel x-y\parallel=dist
(x,D)=\parallel x-z\parallel$$ then $y=z.$
\begin{definition}
A multivalued mapping $T:X\to CB(X)$ is said to be nonexpansive
provided that
$$H(Tx,Ty)\le \|x-y\|,\quad x,y\in X.$$
\end{definition}
  We state Suzuki's condition for multivalued mappings as follows:
\begin{definition}
A multivalued mapping $T:X\to CB(X)$ is said to satisfy the
condition $(C)$ provided that
$$ \frac{1}{2}dist(x,Tx)\le \|x-y\|\implies H(Tx,Ty)\le \|x-y\|,\quad
x,y\in X.$$ \end{definition} We now state the multivalued analogs of
the conditions (E) and $(C_{\lambda})$
 in the following manner (see also \cite{a3}):
\begin{definition}A multivalued mapping $T:X\to CB(X)$ is said to satisfy
 condition $(E_{\mu})$ provided that
$$dist(x,Ty)\le\mu dist(x,Tx)+\|x-y\|,\quad
x,y\in X.$$ We say that $T$ satisfies condition (E) whenever $T$
satisfies $(E_{\mu})$ for some $\mu\geq1.$
\end{definition}
\begin{definition}
A multivalued mapping $T:X\to CB(X)$ is said to satisfy condition
 $(C_{\lambda})$ for some $\lambda\in (0,1)$  provided that
$$\lambda dist(x,Tx)\le \|x-y\| \implies H(Tx,Ty)\le \|x-y\|,\quad
x,y\in X.$$ \end{definition}
\begin{lemma}
 Let $T:X\to CB(X)$ be a multivalued nonexpansive mapping, then $T$ satisfies the condition $(E_{1})$.
\end{lemma}
We now provide an example of a generalized nonexpansive multivalued
mapping satisfying the conditions $(C_{\lambda})$ and
(E) which is not a nonexpansive multivalued mapping (for details, see \cite{a3}).\\
\noindent{\bf Example}. We consider a multivalued mapping $T$ on
$[0,5]$ given by
$$T(x)=\begin{cases}[0,\frac{x}{5}], &x\neq 5\\ \{1\} &x=5.\end{cases}$$
This mapping has the required properties. Finally we recall the
following lemma from \cite{gk}.
\begin{lemma} Let $\{z_n\}$ and $\{w_n\}$ be two bounded sequences
in a Banach space $X$, and let $0<\lambda <1$. If for every natural
number $n$ we have $z_{n+1}=\lambda w_n +(1-\lambda)z_n$ and
$\|w_{n+1}-w_n\|\le \|z_{n+1}-z_n\|$, then $\lim_{n\to
\infty}\|w_n-z_n\|=0$.
\end{lemma}

\section{Common fixed point}
Let $D$ be a nonempty  bounded closed convex subset of a Banach
space $X$ and $\{x_n\}$ a bounded sequence in $X$, we use $
r(x,\seq)$ and $A(D, \seq )$  to denote the asymptotic radius and
the asymptotic center of $\{x_n\}$ in $D$, respectively, i.e.
$$r(D,\seq )=\inf\{ \limsup_{n\to \infty}{\|x_n-x\|}:\, x\in D\},$$
$$A(D,\seq )=\{ x\in D:\,\limsup_{n\to \infty}{\|x_n-x\|}=r(D,\seq )\}.$$
Obviously, the convexity of $E$ implies that $A(D, \seq )$ is
convex. It is known that in a UCED Banach space $X$, the asymptotic
center of a sequence with respect to a weakly compact convex set is
a singleton; the same is true for a sequence in a bounded closed
convex subset of uniformly convex Banach space $X$ \cite{kk}.
\begin{definition}A bounded sequence $\seq$ is said
to be \emph{regular} with respect to $D$ if for every subsequence
$\{x_n^\prime\}$ we have
$$r(D, \seq )=r(D, \{x_n^\prime \});$$ further, $\{x_n\}$ is called
asymptotically uniform relative to $D$
if $$A(D, \seq )=A(D, \{x_n^\prime \}).$$
\end{definition}
The following lemma was proved by Goebel and Lim.
\begin{lemma}(see \cite{gob} and \cite{lim}). Let $\seq$ be a bounded
 sequence in $X$ and let $D$ be a nonempty
closed convex subset of $X$.\begin{enumerate}\item[ (i)] then there
exists a subsequence of $\seq$ which is regular relative to
$D$.\item[(ii)] if $D$ is separable, then $\{x_n\}$ contains a
subsequence which is asymptotically uniform relative to
$D$.\end{enumerate}\end{lemma}
\begin{theorem}
Let $D$ be a nonempty closed convex bounded subset of a
 Banach space $X$. Let $t:D \to D$ be a single valued
 mapping satisfying the conditions (E) and $(C_{\lambda})$
for some  $\lambda\in(0,1)$.  Suppose the asymptotic center relative
$D$ of each sequence in $D$ is nonempty and compact. Then $T$ has a
fixed point. \end{theorem}
\begin{proof}
By Theorem 2.4, there exists a sequence $\{x_{n}\}$ in $ D$ such
that $$\lim_{n\to\infty} dist (x_{n},T{x_{n}})=0.$$  Since $
A(D,\{x_n\})$ is  nonempty convex and compact, by invoking Theorem
2.4 again, there exists a sequence $\{z_{n}\}$ such that
$\lim_{n\to\infty} dist (z_{n},T{z_{n}})=0.$ By passing to
subsequence we can assume that $\lim_{n\to\infty}z_{n}=z.$ By
condition $E$, there exists $\mu>1$ such that $$\|z_{n}-Tz\|\leq
\mu\|z_{n}-Tz_{n}\|+\|z_{n}-z\|.$$  Taking limit in the above
inequality we obtain $Tz=z.$
\end{proof}
\begin{definition}
Let $D$ be a nonempty closed convex bounded subset of a Banach space
$X$, and let $t:D \to X$ and $T:D\to CB(X)$ be two mappings. Then
$t$ and $T$ are said to be commuting if for every $x,y\in D$ such
that $x\in Ty$ and $ty\in D$, we have $tx\in Tty$.
\end{definition}
We now state and prove the main result of this paper.
\begin{theorem}
Let $D$ be a nonempty closed convex bounded subset of a strictly
convex Banach space $X$, $t:D \to D$ be a single valued mapping, and
$T:D\to KC(D)$ be a multivalued mapping. Assume that both mappings
satisfy the conditions (E) and $(C_{\lambda})$ for some
$\lambda\in(0,1)$, and that $t,\, T$ commute. If the asymptotic
center relative $Fix(t)$ of each sequence in $Fix(t)$ is nonempty
and compact, then there exists a point $z\in D$ such that $z=t(z)\in
T(z).$\end{theorem}
\begin{proof}
By Theorem 3.3 the mapping $t$ has a nonempty fixed  point set
$Fix(t)$ which is a closed convex subset of $X$ (by Lemma 2.5). We
show that for $x\in Fix(t)$, $Tx \cap Fix(t)\neq\emptyset$. To see
this, let $x\in Fix(t)$, since $t$ and $T$ commute, we have $ty\in
Tx$ for each $y\in Tx$. Therefore $Tx$ is invariant under $t$ for
each $x\in Fix(t)$. Since $Tx$ is a nonempty compact convex subset
of a Banach space $X$, the asymptotic center in $Tx$ of each
sequence is nonempty and compact. Therefore by Theorem 3.3 we
conclude that $t$ has a fixed point in $Tx$ and therefore $Tx \cap
Fix(t)\neq\emptyset$ for $x\in Fix(t).$\par Now we find an
approximate fixed point sequence for $T$ in $Fix(t)$. Take $x_0\in
Fix(t)$, since $Tx_{0} \cap Fix(t)\neq\emptyset$, we can choose
$y_{0}\in Tx_{0} \cap Fix(t)$. Define
$$x_1=(1- \lambda) x_0+\lambda y_0.$$ Since $Fix(t)$ is a convex set, we have
$x_{1}\in Fix(t)$. Let $y_1\in T(x_1)$ be chosen in such a way that
$$\| y_0-y_1\|=dist(y_0, T(x_1)).$$ We see that $y_{1}\in Fix(t).$
Indeed, we have $$\lambda\| y_{0}-ty_{0}\|=0\leq \|y_{0}-y_{1}\|.$$
Since $t$ satisfies the condition $(C)$, we get
$$\|y_{0}-ty_{1}\|=\|ty_{0}-ty_{1}\|\leq\| y_{0}-y_{1}\|$$ which
contradicts the uniqueness of $y_{1}$ as the unique nearest point of
$y_{0}$ (note that $ty_{1} \in Tx_{1}$). Similarly, put
$$x_2=(1- \lambda) x_1+\lambda y_1,$$ again we choose $y_2\in T(x_2) $ in such a way
that
$$\|y_1-y_2\|=dist(y_1, T(x_2)).$$ By the same argument,
we get $y_{2}\in Fix(t).$ In this way we will find a sequence
$\{x_{n}\}$ in $Fix(t)$ such that $$x_{n+1}= (1-\lambda) x_n+\lambda
y_n.$$ where $y_n\in T(x_n) \cap Fix(t)$ and
$$\| y_{n-1}-y_n\| =dist (y_{n-1}, T(x_n)).$$
Therefore for every natural number $n\ge 1$ we have
$$\lambda\|x_n-y_n\|=\|x_n-x_{n+1}\|$$ from which it follows that
$$\lambda\,
dist(x_n,T(x_n))\le \lambda\|x_n-y_n\|=\|x_n-x_{n+1}\|,\quad n\ge
1.$$ Our assumption now gives
$$H(T(x_n), T(x_{n+1}))\le \|x_n-x_{n+1}\|,\quad n\ge 1,$$
hence for each $n\ge 1$ we have
\begin{align*}\|y_n-y_{n+1}\|=dist(y_n,T(x_{n+1}))&\le H(T(x_n),
T(x_{n+1}))\\ &\le \|x_n-x_{n+1}\|.\end{align*} We now apply Lemma
2.11 to conclude that $\lim_{n\to \infty}\|x_n-y_n\|=0$ where
$y_n\in T(x_n)$. From Lemma 3.2 by passing to a subsequence, we may
assume $\{x_n\}$ is regular asymptotically uniform relative to
$Fix(t)$. Let $r=r(Fix(t),\{x_n\})$. Now, we show that $ Tx\cap
A(Fix (t),\{x_n\})\neq\emptyset$ for $x\in A(Fix(t),\{x_n\})$. The
compactness of $T(x_n)$ implies that for each $n$
 we can take $y_n\in T(x_n)$ such that
$$\|x_n-y_n\|=dist(x_n, T(x_n)).$$ Suppose $x\in
A(Fix(t),\{x_n\})$. Since $T(x)$ is compact, for each $n$, we choose
$z_n\in T(x)$ such that $$\|z_n-y_n\|=dist(y_n, T(x)).$$ By
assumption there exist $\mu>1$ such that \begin{align*}\|y_n-z_n\|=
dist (y_{n},Tx) &\le H(T(x_n), T(x))\\ &\le \mu\, dist(x_n,
T(x_n))+\|x_n-x\|.\end{align*}
 Since $T(x)$ is compact, the sequence $\{z_n\}$ has a convergent
subsequence $\{z_{n_k}\}$ with $\lim_{k\to \infty}z_{n_k}=z\in
T(x)$. Note that$$\|x_{n_k}-z\|\le
\|x_{n_k}-y_{n_k}\|+\|y_{n_k}-z_{n_k}\|+\|z_{n_k}-z\|$$ This entails
$$\limsup _{k\to\infty}\| x_{n_k}-z\|\le \limsup_{k\to
\infty}\|x_{n_k}-x\|\le r.$$ Since $\{x_n\}$ is regular
asymptotically uniform relative to $Fix(t)$, this shows that
 $z\in A(Fix(t),\{x_{n_k}\})=A(Fix(t),\{x_n\})$
  therefore $$z\in Tx\cap
A(Fix(t),\{x_n\}).$$ Now, we show that there exists a sequence
$\{z_{n}\}\subset A(Fix(t),\{x_n\})$ such that  $\lim
dist_{n\to\infty}(z_n,Tz_n)=0$. Indeed, take $z_0\in
A(Fix(t),\{x_n\})$, since $Tz_0\cap A(Fix(t),\{x_n\})\neq
\emptyset$, there exists $y_0\in Tz_0\cap A(Fix(t),\{x_n\}).$ We
define
$$z_1=(1- \lambda) z_0+\lambda y_0.$$
Since $A(Fix(t),\{x_n\})$ is convex, we have $z_1\in A(Fix(t),\{x_n\}).$\\
Similarly, since $Tz_1\cap A(Fix(t),\{x_n\})\neq \emptyset$, by the
compactness of $Tz_1\cap A(Fix(t) ,\{x_n\})$ we can choose $y_1\in
Tz_1\cap A(Fix(t),\{x_n\})$ in such a way that
$$\|y_0-y_1\|=dist(y_0, Tz_1\cap A(Fix(t),\{x_n\})).$$
 In this way we will find a sequence $\{z_n\}\in A(Fix(t),\{x_n\})$
such that
$$z_{n+1}=(1- \lambda) z_n+\lambda y_n$$ where $y_n\in Tz_n\cap
A(Fix(t),\{x_n\})$ and $$ \|y_{n-1}-y_n\|=dist(y_{n-1}, Tz_n\cap
A(Fix(t),\{x_n\})).$$ Therefore for every natural number $n\ge 1$ we
have
$$\lambda\|z_n-y_n\|=\|z_n-z_{n+1}\|$$ from which it follows that
$$\lambda\,
dist(z_n,T(z_n))\le \lambda\|z_n-y_n\|=\|z_n-z_{n+1}\|,\quad n\ge
1.$$ Our assumption now gives
$$H(T(z_n), T(z_{n+1}))\le \|z_n-z_{n+1}\|,\quad n\ge 1,$$
hence
$$\|y_n-y_{n+1}\|=dist(y_n,T(z_{n+1}))\le H(T(z_n), T(z_{n+1}))\le
\|z_n-z_{n+1}\|,\quad n\ge 1.$$ We now apply Lemma 2.11 to conclude
that $\lim_{n\to \infty}\|z_n-y_n\|=0$ where $y_n\in T(z_n)$.

Since $z_n\in  A(Fix(t),\{x_n\})$, and $A(Fix(t),\{x_n\})$ is
compact, by passing to subsequence we may assume that $z_n$ is
convergent to $ z\in A(Fix(t),\{x_n\})$  as $n\to\infty$. Since $Tz$
is compact, for each $n\ge 1$, we can choose $ w_{n}\in Tz$ such
that $\|w_n-z_n\|=dist(z_n, Tz).$ Moreover $w_{n}\in Fix(t)$ for all
natural numbers $n\geq1.$ Indeed, since
$$ \lambda\|z_{n}-tz_{n}\|=0\leq \|w_{n}-z_{n}\|,\qquad n\geq1,$$ we have
$$\|z_{n}-tw_{n}\|=\|tw_{n}-tz_{n}\|\leq\|w_{n}-z_{n}\|.$$ Since
$z\in Fix(t)$ and $w_{n}\in T(z)$, by the fact that the mappings $t$
and $T$ commute, we obtain $tw_{n}\in Ttz=Tz$. Now, by the
uniqueness of $w_{n}$ as the nearest point to $z_{n}$, we get
$tw_{n} =w_{n}\in Fix(t).$ \\Since $ Tz$ is compact the sequence
$\{w_n\}$ has a convergent subsequence $\{w_{n_k}\}$ with
$\lim_{k\to \infty}w_{n_k}=w\in Tz $. Because $w_{n_{k}}\in Fix(t)$
for all n, and $Fix(t)$ is closed, we obtain $w\in Fix(t)$. By
assumption there exists $\mu>1$  such that
$$dist(z_{n_{k}}, Tz) \leq \mu
dist(z_{n_{k}},T(z_{n_{K}}))+\|z_{n_{k}}-z\|.$$ Note
that\begin{multline*}\|z_{n_k}-w\|\le
\|z_{n_k}-w_{n_k}\|+\|w_{n_k}-w\|\\\leq \mu
dist(z_{n_{k}},T(z_{n_{k}}))+\|w_{n_{k}}-w\|
+\|z_{n_k}-z\|.\end{multline*}
 This entails $$\limsup _{k\to\infty}\| z_{n_k}-w\|\le \limsup_{k\to
\infty}\|z_{n_k}-z\|.$$ We conclude that $z=w$, hence $z=tz\in T z.$

\end{proof}
\begin{theorem}
Let $D$ be a nonempty closed convex bounded subset of a uniformly
convex Banach space $X$, $t:D \to D$, and $T:D\to KC(D)$ a single
valued and a multivalued  mapping, both satisfying the conditions
(E) and $(C_{\lambda})$ for some $\lambda\in(0,1)$. Assume that
$t,\, T$ commute. Then there exists a point $z\in D$ such that
$z=t(z)\in T(z)$ .\end{theorem}

\begin{proof}
By Theorem 3.3 the mapping $t$ has a nonempty fixed point set
$Fix(t)$ which is a closed convex subset of $X$ (by Lemma 2.5).
Since $X$ is uniformly convex, we have asymptotic center relative
$Fix(t)$ of each sequence in $Fix(t)$ is nonempty and singleton
(hence compact). Therefore, by theorem 3.5, $T$ has a fixed point.
\end{proof}
\begin{theorem}
Let $D$ be a nonempty weakly compact convex subset of  a UCED Banach
space $X$. Let $t:D \to D$, and $T:D\to KC(D)$ be a single valued
and a multivalued  mapping respectively, both satisfying the
conditions (E) and $(C_{\lambda})$ for some $\lambda\in(0,1)$.
Assume that $t,\, T$ commute. Then there exists a point $z\in D$
such that $z=t(z)\in T(z)$ .\end{theorem}
\begin{proof}
By Theorem 3.3 the mapping $t$ has a nonempty fixed point set
$Fix(t)$ which is a closed convex subset of $X$ (by Lemma 2.5).
Since $D$ is weakly compact we conclude that $Fix(t)$ is weakly
compact. Since $X$ is UCED, we conclude that the asymptotic center
relative $Fix(t)$ of each sequence in $Fix(t)$ is nonempty and
singleton (hence compact). Therefore, by Theorem 3.5, $T$ has a
fixed point.\end{proof}

\begin{theorem}
Let $D$ be a nonempty compact convex subset of a strictly convex
Banach space $X$, $t:D \to D$, and $T:D\to KC(D)$ a single valued
and a multivalued  mapping, both satisfying the conditions (E) and
$(C_{\lambda})$ for some  $\lambda\in(0,1)$. Assume that $t,\, T$
commute. Then there exists a point $z\in D$ such that $z=t(z)\in
T(z).$\end{theorem}
\begin{proof}
By Theorem 3.3 the mapping $t$ has a nonempty fixed point set
$Fix(t)$ which is a closed convex subset of $X$ (by Lemma 2.5).
Since $D$ is compact we conclude that $Fix(t)$ is  compact. Since
$X$ is strictly convex, we infer that the asymptotic center relative
$Fix(t)$ of each sequence in $Fix(t)$ is nonempty and compact.
Therefore, by Theorem 3.5, $T$ has a fixed point.\end{proof}
\begin{theorem}
Let $D$ be a nonempty closed convex bounded subset of a strictly
convex Banach space $X$. Let $t:D \to D$, and $T:D\to KC(D)$ be two
nonexpansive mappings. Assume that $t,\, T$ commute.  Suppose the
asymptotic center relative $Fix(t)$ of each sequence in $Fix(t)$ is
nonempty and compact. Then there exists a point $z\in D$ such that
$z=t(z)\in T(z).$\end{theorem}
\begin{proof} By Lemma 2.10 the mapping $T$
satisfies the condition $E_{1}$. we also note that $T$ satisfies the
condition $(C_{\lambda})$ for all $\lambda\in(0,1).$ So the result
follows from Theorem 3.5\end{proof}
\begin{corollary}
Let $D$ be a nonempty closed convex bounded subset of a uniformly
convex Banach space $X$. Let $t:D \to D$, and $T:D\to KC(D)$ be two
nonexpansive mappings. Assume that $t,\, T$ commute. Then there
exists a point $z\in D$ such that $z=t(z)\in T(z)$ .\end{corollary}

\begin{corollary}
Let $E$ be a nonempty weakly compact convex subset of  a UCED Banach
space $X$. Let $t:D \to D$, and $T:D\to KC(D)$ be two nonexpansive
mappings. Assume that $t,\, T$ commute. Then there exists a point
$z\in D$ such that $z=t(z)\in T(z)$ .\end{corollary}
\begin{corollary}
Let $D$ be a nonempty compact convex subset of a strictly convex
Banach space $X$. Let $t:D \to D$, and $T:D\to KC(D)$ be two
nonexpansive mappings. Assume that $t,\, T$ commute. Then there
exists a point $z\in D$ such that $z=t(z)\in T(z).$\end{corollary}

\end{document}